\documentclass[12pt,regno]{amsart}
\usepackage{epsf,epsfig,amsmath}
\usepackage{amssymb,latexsym}
\usepackage{tikz}
\usepackage{amsthm} 
\usepackage{amsfonts}
\usepackage{graphicx,psfrag}

\numberwithin{equation}{section}

\newtheorem{thm}{Theorem}[section]
\newtheorem{pro}[thm]{Proposition}
\newtheorem{lem}[thm]{Lemma}

\newtheorem{cor}[thm]{Corollary}

\theoremstyle{definition}
\newtheorem{defi}[thm]{Definition}

\theoremstyle{remark}
\newtheorem{obs}[thm]{Observation}
\newtheorem{rem}[thm]{Remark}
\theoremstyle{definition}
\newtheorem*{ack}{Acknowledgement}

\title[Special matchings in Coxeter groups]{Special matchings in Coxeter groups}

\author{Fabrizio Caselli}\author{Mario Marietti}

\address{Dipartimento di matematica, Universit\`a di Bologna, Piazza di Porta San Donato 5, 40126 Bologna, Italy}
\address{Dipartimento  di Ingegneria Industriale e Scienze Matematiche, Universit\`a Politecnica delle Marche, Via Brecce Bianche, 60131 Ancona,  Italy}

\email{fabrizio.caselli@unibo.it}
\email{m.marietti@univpm.it}


\begin{document}

\maketitle

\begin{abstract}

Special matchings are purely combinatorial objects associated with a partially ordered set, which have applications in Coxeter group theory. We provide an explicit characterization and a complete classification of all special matchings of any lower Bruhat interval. The results hold in any arbitrary Coxeter group and have also applications in the study of the corresponding parabolic Kazhdan--Lusztig polynomials.
\end{abstract}

\section{Introduction}

Coxeter groups have a wide range of applications in  several areas of mathematics such as algebra, geometry, and combinatorics. The Bruhat order plays an important role in Coxeter group theory; it was introduced,  in the case of Weyl groups, as the partial order structure controlling the inclusion between the Schubert varieties, but it is prominent  also in other contexts, including the study of Kazhdan--Lusztig polynomials. For Coxeter group theory and its applications, we  refer the  reader to the books \cite{BB}, \cite{Bou}, \cite{Hil}, \cite{Hum} (and references cited there).

Special matchings are  purely combinatorial objects, which can be defined for any partially ordered set, and have their main applications in  
Coxeter group theory.  The special matchings of a Coxeter group are abstractions of the maps  
given by the multiplication (on the left or on the right) by a Coxeter generator. Precisely, let $(W,S)$ be a Coxeter system, so that $W$ is both a group and a partially ordered set (under Bruhat order), and let  $e$ denote the identity element of $W$. Given $w\in W$,
a special matching of $w$  is an involution \( M:[e,w]\rightarrow [e,w] \)  of the Bruhat interval $[e,w]$
such that 
\begin{enumerate}
\item either \( u \lhd M(u)\) or  \( u \rhd M(u)\), for all \( u\in [e,w]  \),
\item  if $u_1\lhd u_2$ then  $M(u_1)\leq M(u_2),$
 for all \( u_1,u_2\in [e,w] \) such that \( M(u_1)\neq u_2 \).
\end{enumerate}
(Here, $\lhd$ denotes the covering relation, i.e., $x\lhd y$ means that $x<y$ and there is no $z$ with $x<z<y$.)

Special matchings were introduced in \cite{Brepr} and there studied for  the symmetric group (the prototype of a Coxeter group).
For sake of completeness, we also recall that Bruhat intervals are Eulerian posets and that a different but equivalent construction was introduced  by Du Cloux for all Eulerian posets in \cite{DCEJC}.
Later, for any arbitrary Coxeter group $W$, special matchings have been shown to be crucial in the study of the Kazhdan--Lusztig polynomials  of $W$ (see \cite{BCM1}), the Kazhdan--Lusztig representations of $W$ (see \cite{BCM2}), and the poset-theoretic properties of $W$ (see \cite{MJaco}). In particular, the main result
 in \cite{BCM1} is a formula to compute the  Kazhdan--Lusztig polynomial $P_{u,v}$, $u\leq v\in W$, from the knowledge of the special matchings of the elements in  $[e,v]$; as a corollary, $P_{u,v}$ depends only on the isomorphism class of the interval $[e,v]$.   

The main result of this paper is a complete classification of special matchings of lower Bruhat intervals in arbitrary Coxeter groups. In the process of proving such classification we provide several partial results on the structure of special matchings which have been applied in the theory of parabolic Kazhdan--Lusztig polynomials, which are a generalization of the Kazhdan--Lusztig polynomials  introduced by Deodhar in \cite{Deo87}. In fact, since the appearance of \cite{BCM1}, the authors have been asked many times whether the results in it could be generalized to the parabolic Kazhdan--Lusztig polynomials setting. This problem, still open for a  general Coxeter group $W$, has been recently solved in  \cite{Mpreprint} for the doubly laced Coxeter groups (and, also, in the much easier case of dihedral Coxeter systems) and in \cite{Tel} for universal Coxeter groups. Both of these papers make use of the classification of special matchings given in the present work.
 
 Since the results in this paper are valid for all Coxeter groups, we believe that they might be useful to extend the results in  \cite{Mpreprint} and \cite{Tel} also for other classes of Coxeter groups.

\section{Notation, definitions and preliminaries}

In this section, we collect  some notation, definitions,
 and results that will be used in the rest of this work.

We  follow \cite[Chapter 3]{StaEC1} for undefined notation and 
terminology concerning partially ordered sets. In particular,
 given $x,y$ in a partially ordered set  $P$, 
we say that $y$ {\em covers} $x$ and we write $x \lhd y$ if the interval $[x,y] = \{ z \in P: \; x \leq z \leq y \}$ has two elements,  $x$ and $y$.
We say that a poset $P$ is {\em graded} if $P$ has a minimum $\hat 0$ and there is a function
$\rho : P \rightarrow \mathbb N$ such that $\rho (\hat{0})=0$ and $\rho (y) =\rho (x)
+1$ for all $x,y \in P$ with $x \lhd y$. 
(This definition is slightly different from the one given in \cite{StaEC1}, but is
more convenient for our purposes.) 
We then call $\rho $ the {\em rank function}
of $P$.
The {\em Hasse diagram} of $P$ is any drawing of the graph having $P$ as vertex set and $ \{ \{ x,y \} \in \binom {P}{2} : \text{ either $x \lhd y$ or $y \lhd x$} \}$ as edge set, with the convention that, if $x \lhd y$, then the edge $\{x,y\}$ goes \emph{upward} from $x$ to $y$. When no confusion arises we will make no distinction between the Hasse diagram and its underlying graph.

A {\em matching} of a poset \( P \) is an involution
\( M:P\rightarrow P \) such that \( \{v,M(v)\}\) is an edge in the Hasse diagram of $P$, for all \( v\in P \).
A matching \( M \) of \( P \) is {\em special} if\[
u\lhd v\Longrightarrow M(u)\leq M(v),\]
 for all \( u,v\in P \) such that \( M(u)\neq v \).

The two simple results in the following lemma will often be used without explicit mention (see  \cite[Lemmas 2.1 and 4.1]{BCM1}). Given a poset $P$, two matchings $M$ and $N$ of  $P$, and $u\in P$, we denote by  $\langle M, N \rangle (u) $ the orbit of $u$ under the action of the subgroup $\langle M,N\rangle $ of the symmetric group on $P$ generated
 by $M$ and $N$. We call an interval $[u, v]$ in a poset $P$ {\em dihedral} if it is isomorphic to an interval in a finite Coxeter system of rank 2 ordered by Bruhat order (see Figure \ref{diheinte}).
\begin{lem}
\label{res}
Let \( P \) be a finite graded poset. 
\begin{enumerate}
\item Let $M$ be a special matching of $P$, and $u,v\in P$ be such that $u\leq v$,
$M(v)\lhd v$ and $M(u)\rhd u$. Then $M$ restricts to a special matching of the interval $[u,v]$.
\item Let  $M$ and $N$ be two special matchings of $P$. Then, for all $u \in P$, the orbit $\langle M, N \rangle (u) $ is a dihedral interval.
\end{enumerate} 
\end{lem}
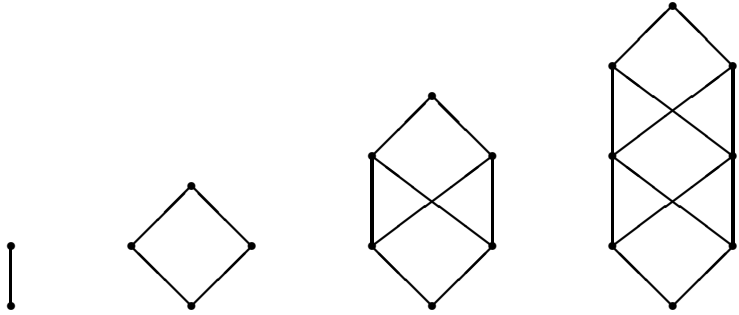
\begin{figure}
 \setlength{\unitlength}{8mm}
\begin{center}
\begin{picture}(12,6)
\thicklines
\put(0,0){\line(0,1){1}}
\put(0,0){\circle*{0.15}}
\put(0,1){\circle*{0.15}}

\put(3,0){\line(1,1){1}}
\put(3,0){\line(-1,1){1}}
\put(4,1){\line(-1,1){1}}
\put(2,1){\line(1,1){1}}
\put(3,0){\circle*{0.15}}
\put(3,2){\circle*{0.15}}
\put(2,1){\circle*{0.15}}
\put(4,1){\circle*{0.15}}

\put(7,0){\line(1,1){1}}
\put(7,0){\line(-1,1){1}}
\put(8,1){\line(0,1){1.5}}
\put(6,1){\line(0,1){1.5}}
\put(6,2.5){\line(1,1){1}}
\put(8,2.5){\line(-1,1){1}}
\put(6,1){\line(4,3){2}}
\put(8,1){\line(-4,3){2}}
\put(7,0){\circle*{0.15}}
\put(8,1){\circle*{0.15}}
\put(6,1){\circle*{0.15}}
\put(6,2.5){\circle*{0.15}}
\put(8,2.5){\circle*{0.15}}
\put(7,3.5){\circle*{0.15}}

\put(11,0){\line(1,1){1}}
\put(11,0){\line(-1,1){1}}
\put(12,1){\line(0,1){3}}
\put(10,1){\line(0,1){3}}
\put(10,1){\line(4,3){2}}
\put(12,1){\line(-4,3){2}}
\put(10,2.5){\line(4,3){2}}
\put(12,2.5){\line(-4,3){2}}
\put(10,4){\line(1,1){1}}
\put(12,4){\line(-1,1){1}}
\put(11,0){\circle*{0.15}}
\put(12,1){\circle*{0.15}}
\put(10,1){\circle*{0.15}}
\put(12,2.5){\circle*{0.15}}
\put(10,2.5){\circle*{0.15}}
\put(10,4){\circle*{0.15}}
\put(12,4){\circle*{0.15}}
\put(11,5){\circle*{0.15}}

\end{picture}
\end{center}
\caption{\label{diheinte} Dihedral intervals of rank 1,2,3,4}
\end{figure}

We follow \cite{BB} for undefined Coxeter groups notation
and terminology. 

Given a Coxeter system $(W,S)$ and $s,r\in S$, we denote by $m_{s,r}$ the order of the product $sr$. Given $w \in W$, we denote by $\ell
(w )$ the length of $w $ with respect to $S$,
and we let
\begin{align*}
  D_{R}(w )& =  \{ s \in S : \;
\ell(w s) < \ell(w ) \}, \\
D_{L}(w )& = \{ s \in S: \; \ell(sw)<\ell(w)\}. 
\end{align*}

We call the elements of $ D_R(w )$ and $D_{L}(w )$, respectively, the {\em right descents} and the 
{\em left descents} of $w $.
We denote by $e$
the identity of $W$, and we let $T =\{ w s w ^{-1} : w \in W, \; s \in S \}$ be the set of {\em reflections} of $W$.

The Coxeter group $W$ is partially ordered by {\em Bruhat order} (see, e.g.,  \cite[\S 2.1]{BB} or \cite[\S 5.9]{Hum}), which will be denoted by $\leq$. The Bruhat order is the partial order whose  covering relation $\lhd$ is as follows: given $u,v \in W$, we have $u \lhd v$ if and only if $u^{-1}v \in T$ and $\ell (u)=\ell (v)-1$. 
There is a well known characterization of Bruhat order
on a Coxeter group (usually referred to as the {\em Subword Property})
that we will use repeatedly in this work,
often without explicit mention. We recall it here for the reader's
convenience. 

By a {\em subword} of a word $s_{1}\textrm{-}s_{2} \textrm{-} \cdots \textrm{-} s_{q}$ (where we use the symbol ``$\textrm{-}$'' to separate letters in a word in the alphabet $S$) we mean
a word of the form
$s_{i_{1}}\textrm{-} s_{i_{2}}\textrm{-} \cdots \textrm{-} s_{i_{k}}$, where $1 \leq i_{1}< \cdots
< i_{k} \leq q$.
If $w\in W$ then a {\em reduced expression} for $w$ is a word $s_1\textrm{-}s_2\textrm{-}\cdots \textrm{-}s_q$ such that $w=s_1s_2\cdots s_q$ and $\ell(w)=q$. When no confusion arises we also say in this case that $s_1s_2\cdots s_q$ is a reduced expression for $w$. 
\begin{thm}[Subword Property]
\label{subword}
Let $u,w \in W$. Then the following are equivalent:
\begin{itemize}
\item $u \leq w$ in the Bruhat order,
\item  every reduced expression for $w$ has a subword that is 
a reduced expression for $u$,
\item there exists a  reduced expression for $w$ having a subword that is 
a reduced expression for $u$.
\end{itemize}
\end{thm}
A proof of the preceding result can be found, e.g., in \cite[\S 2.2]{BB} or \cite[\S 5.10]{Hum}. 
It is well known that $W$, partially ordered by Bruhat order, is a graded poset having $\ell$ as its rank function.

We recall that two reduced expressions of an element are always linked by a sequence of {\em braid moves}, where a braid move consists in substituting a factor $s\textrm{-}t\textrm{-}s\textrm{-} \cdots$ ($m_{s,t}$ letters) with a factor $t\textrm{-}s\textrm{-}t\textrm{-} \cdots$ ($m_{s,t}$ letters), for some $s,t\in S$. We also recall that, if $w\in W$ and $s,t\in D_R(w)$, then there exists a reduced expression for $w$ of the form $s_1\textrm{-}\cdots \textrm{-} s_k\textrm{-}\underbrace{s\textrm{-}t\textrm{-}s\textrm{-} \cdots}_{m_{s,t}\textrm { letters}}$.

For each subset $J\subseteq S$, we denote by  $W_J $ the parabolic subgroup of $W$ generated by $J$, and by  $W^{J}$ the set of minimal coset representatives:
$$W^{J} =\{ w \in W \, : \; D_{R}(w)\subseteq S\setminus J \}.$$
The following is a useful factorization of $W$ (see, e.g., in \cite[\S 2.4]{BB} or \cite[\S 1.10]{Hum}).
\begin{pro}
\label{fattorizzo}
Let $J \subseteq S$. Then:
\begin{enumerate}
\item[(i)] 
every $w \in W$ has a unique factorization $w=w^{J} \cdot w_{J}$ 
with $w^{J} \in W^{J}$ and $w_{J} \in W_{J}$;
\item[(ii)] for this factorization, $\ell(w)=\ell(w^{J})+\ell(w_{J})$.
\end{enumerate}
\end{pro}
There are, of course, left versions of the above definition and
result. Namely, if we let  
\begin{equation*}
^{J}  W = \{ w \in W \, : \; D_{L}(w)\subseteq S\setminus J \} =(W^{J})^{-1}, 
\end{equation*}
then every $w \in W$ can be uniquely factorized $
w={_{J} w} \cdot  {^{J}\! w}$, where ${_{J} w} \in W_{J}$, ${^{J} \!
w} \in \,  ^{J} W$,
and $\ell(w)=\ell({_{J} w })+\ell({^{J} \! w})$.

We will also need the two following well known results (a proof of the first can be found, e.g., in \cite[Lemma 7]{Hom74}, while the second is easy to prove).
\begin{pro}
\label{unicomax}
Let \( J\subseteq S \) and $w \in W$.  The set 
$ W_{J}\cap [e,w] $ has a unique  maximal element  \( w_0(J) \),
so that \( W_{J}\cap [e,w]$ is the interval $[e,w_0(J)] \). 
\end{pro}

We note that the term $w_0(J)$ denotes something different in \cite{BB} and that if $J=\{s,t\}$ we adopt the lighter notation $w_0(s,t)$ to  mean $w_0(\{s,t\})$.

\begin{pro}
\label{mantiene}
Let \( J\subseteq S \) and $v,w \in W$, with $v\leq w$. Then $v^J\leq w^J$ and ${^J\! v}\leq \,{^J\!w}$.
\end{pro}

The following is a useful combinatorial property fulfilled by all  Coxeter groups (see \cite[Proposition 3.2]{BCM1}).
\begin{pro}
\label{k32}
A Coxeter group $W$ avoids $K_{3,2}$ configurations (see Figure \ref{K32fig}), which means that there are no elements 
$a_1,a_2,a_3,b_1,b_2 \in W$, all distinct, such that either 
$a_i \lhd b_j$ for all $i \in [3]$, $j \in [2]$ or 
$a_i \rhd b_j$ for all $i \in [3]$, $j \in [2]$. 
\end{pro}

\begin{figure}
 \setlength{\unitlength}{10mm}
\begin{center}
\begin{picture}(12,3)
\thicklines
\put(0,0){\line(1,2){1}}
\put(0,0){\line(3,2){3}}
\put(2,0){\line(-1,2){1}}
\put(2,0){\line(1,2){1}}
\put(4,0){\line(-3,2){3}}
\put(4,0){\line(-1,2){1}}
\put(0,0){\circle*{0.15}}
\put(1,2){\circle*{0.15}}
\put(2,0){\circle*{0.15}}
\put(3,2){\circle*{0.15}}
\put(4,0){\circle*{0.15}}

\put(8,0){\line(-1,2){1}}
\put(8,0){\line(1,2){1}}
\put(8,0){\line(3,2){3}}
\put(10,0){\line(-3,2){3}}
\put(10,0){\line(-1,2){1}}
\put(10,0){\line(1,2){1}}
\put(8,0){\circle*{0.15}}
\put(10,0){\circle*{0.15}}
\put(7,2){\circle*{0.15}}
\put(9,2){\circle*{0.15}}
\put(11,2){\circle*{0.15}}

\end{picture}
\end{center}
\caption{\label{K32fig} $K_{3,2}$ configurations}
\end{figure}
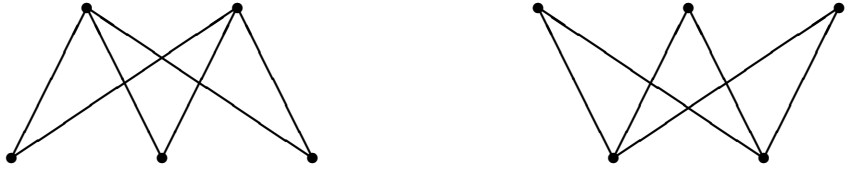

We are interested in the special matchings of a Coxeter group $W$ (to be precise, of intervals in $W$)  partially ordered by  Bruhat order. Given $w\in W$, we say that $M$ is a matching of $w$ if $M$ is a matching of the lower Bruhat interval $[e,w]$.
If \( s\in D_{R}(w) \) (respectively, $s\in D_{L}(w)$) we define a matching 
\(\rho_{s}  \) (respectively, $\lambda_{s}$) of $w$ 
by \( \rho_{s}(u)=us \) (respectively, 
$\lambda_{s}(u)=su$) for all \( u\leq w \). From 
the ``Lifting Property'' (see, e.g., \cite[Proposition 2.2.7]{BB} or \cite[Proposition 5.9]{Hum}), it easily follows that $\rho_s$ 
(respectively, $\lambda_s$) is a special matching of $w$. 
We call a matching $M$ of $w$ a \emph{left multiplication
matching} if there exists \( s \in S \) such that  \( M=\lambda _{s} \)
on $[e,w]$, and we call it a  \emph{right multiplication
matching} if there exists \( s \in S \) such that  \( M=\rho _{s} \)
on $[e,w]$.

The following result is a useful criterion to establish when two special matchings of an element coincide, and in particular it says that a special matching $M$ is uniquely determined by its action on dihedral intervals containing $e$ and $M(e)$. It can be proved following mutatis mutandis the proof of \cite[Lemma 5.2]{BCM1}, whereof this lemma represents a natural generalization. 
\begin{lem}\label{diheequa}
 Let $v,w,w'\in W$ $v\leq w,w'$ and $M, M'$ be special matchings of $w$ and $w'$ respectively such that $M(e)=M'(e)=s$. We also assume that
 \[
  M(u)=M'(u)\,\,\forall u \in \bigcup_{t\in S}[e,v_0(s,t)].
 \]
 Then $M(v)=M'(v)$.
\end{lem}
For the reader's convenience, we write the following  results, which will be needed later (see \cite[Propositions 7.3 and 7.4]{BCM1} for a proof).

\begin{pro}
\label{734}
Given a Coxeter system  $(W,S)$, an element  $w\in W$ such that $w\geq r$ for all $r\in S$, and a  special matching  $M$  of $w$, let $s = M(e)$ and $J=\{r\in S : M(r)=sr\}$. 
\begin{enumerate}
\item For all $u \leq w$, every $r\in J$ such that  $r \leq u^J$  commutes with $s$.
\item Let \( t \in S \) be such that \( M \)
is not a multiplication matching on 
$[e,w_0(s,t)]$. Suppose
that \( M(t)=ts \) and  let \( x_{0}\) be
the minimal element in 
$[e,w_0(s,t)]$
such that \( M(x_{0})\neq x_{0}s \) and $\alpha \in D_L(x_0)$.
Then \( \alpha \nleq (u^{J})^{\{s,t\}} \) for all $u \leq w$.
\end{enumerate}
\end{pro}

One of the main  propedeutic results of \cite{BCM1} is the following  (see \cite[Theorem 7.6]{BCM1}). Let $(W,S)$ be a Coxeter system. Given $u \in W$, $J \subseteq S$, $s \in J$
and $t \in S\setminus J$, we may factorize and write 
$$u=u^J \cdot \, u_J=(u^{J})^{\{s,t\}}\, \cdot  \, (u^{J})_{\{s,t\}} \, \cdot \, _{\{s\}} (u_{J}) \, \cdot \, ^{\{s\}}(u_{J}),$$ (see Proposition \ref{fattorizzo}). Recall that, clearly, if $M$ is a special matching and $M(e)=s$, then $M(r)\in \{rs,sr\}$, for all $r\in S$ such that $M(r)$ is defined.

\begin{thm}[\cite{BCM1}, Lemma 7.1 and Theorem 7.6]
\label{ma16}Let \( (W,S) \) be a Coxeter system, \( w\in W \), 
\( M \) be a special matching of \( w \) and \( s=M(e) \). Set $J=\{r \in S : r \leq w, M(r)=sr\}$.
\begin{itemize}
\item [(i)] Suppose that there exists a (necessarily unique) \( t \in S \) such that $M$ is not a multiplication matching on  $[e,  w_0(s,t)]$. Assume that  \( M(t)=ts \). Then \[
M(u)=(u^{J})^{\{s,t\}}\, \cdot \, M\Big ((u^{J})_{\{s,t\}} \, \cdot \, _{\{s\}} (u_{J})\Big )\, \cdot \,
 ^{\{s\}}(u_{J}),\]
for all $u \leq w$.
\item [(ii)]Suppose that  \( M \) is a multiplication matching on 
 $[e,  w_0(s,x)]$,
for all \( x\in S \). Then \[
M(u)=u^{J}su_{J},\]
for all $u \leq w$.
\end{itemize}
\end{thm}

\section{First algebraic properties of special matchings in Coxeter groups}

Theorem \ref{ma16} establishes fundamental algebraic properties satisfied by the special matchings of  lower Bruhat intervals. These properties provide what is needed  for the study of the Kazhdan--Lusztig polynomials developed in \cite{BCM1}, however they are not sufficient for the generalization of the results in \cite{BCM1} to the 
 parabolic  Kazhdan--Lusztig polynomials. 
In this section, we provide further algebraic properties of the special matchings which are needed in the parabolic setting. These results will also serve as a motivation for the definition of left and right systems in Section \ref{leftright} and the resulting characterization and classification of all special matchings of a lower Bruhat interval. 

We begin with the following easy result, which holds for a larger class of posets, not only for the Coxeter groups. For sake of simplicity, we say that a matching $M$ of a poset $P$ is 
N-avoiding if there are not 2 elements $u,v \in P$, $u\lhd v$, $u\neq M(v)$, such that $u \lhd M(u)$ and $M(v) \lhd v$. We call such configuration the 
N-configuration (see the following picture, where we join two vertices by a double edge if they are mapped to each other by the matching $M$).

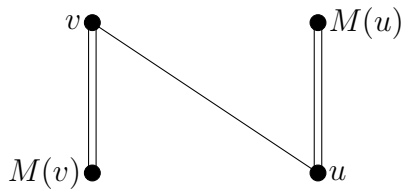
$$
\begin{tikzpicture}

\node[left] at (-1.5,1.5){$M(v)$};
\node[right] at (1.5,1.5){$u$};

\draw{(-1.45,3.5)--(-1.45,1.5)}; 
\draw{(-1.55,3.5)--(-1.55,1.5)}; 
\draw{(1.45,3.5)--(1.45,1.5)}; 
\draw{(1.55,3.5)--(1.55,1.5)}; 
\draw{(-1.5,3.5)--(1.5,1.5)}; 

\draw[fill=black]{(1.5,1.5) circle(3pt)};
\draw[fill=black]{(1.5,3.5) circle(3pt)};
\draw[fill=black]{(-1.5,3.5) circle(3pt)};
\draw[fill=black]{(-1.5,1.5) circle(3pt)};

\node[left] at (-1.5,3.5){$v$};
\node[right] at (1.5,3.5){$M(u)$};

\end{tikzpicture}
$$

\begin{pro}
\label{nonconfigurazione}
Let $P$ be a graded poset such that all its intervals of rank 2 have cardinality $\geq 4$. Then a matching $M$ of $P$ is special if and only if it is {\emph N}-avoiding.
\end{pro}
\begin{proof}
The \lq \lq only if\rq \rq  part is clear. We prove the \lq \lq if\rq \rq part by showing that a matching $M$ which is not special  must contain an N-configuration.

Since $M$ is not special, there exist $x,y\in P$, $x\lhd y\neq M(x)$, such that $M(x) \not \leq M(y)$. We may assume that either
\begin{enumerate}
\item $M(x)\lhd x$, $M(y)\lhd y $, and $M(x) \not \leq M(y)$, or
\item $M(x)\rhd x$, $M(y)\rhd y $, and $M(x) \not \leq M(y)$, 
\end{enumerate}
otherwise $x$ and $y$ form an N-configuration and we are done.
We only treat the first case, the second one being completely similar. 

So suppose that $M(x)\lhd x$, $M(y)\lhd y $, and $M(x) \not \leq M(y)$. Since $[M(x),y]$ is an interval of rank 2, it contains $M(x)$, $x$, $y$, and another element $p$. We have $p\neq M(y)$ since $M(x) \not \leq M(y)$. If $p\lhd M(p)$, we obtain an N-configuration with $p$, $M(p)$, $y$, and $M(y)$; if $p\rhd M(p)$, we obtain an N-configuration with $p$, $M(p)$, $x$, and $M(x)$.
\end{proof}

We note that all rank $2$ intervals in a Coxeter group have cardinality $4$.

\begin{obs}
Let $(W,S)$ be a Coxeter system with Coxeter matrix $M$. Fix an  element $w\in W$. 
By Proposition \ref{unicomax},  the intersection of the lower Bruhat interval $[e,w]$ with the dihedral parabolic subgroup $W_{\{r,r'\}}$  generated by any two given generators $r,r'\in S$ has a maximal element $w_0(r,r')$. 
Let $M'$ be the matrix whose entry $m'_{r,r'}$ is the length  $\ell(w_0(r,r'))$ of the element  $w_0(r,r')$, for each pair $(r,r')\in S\times S$. Let $(W',S)$ be the Coxeter system with the same set $S$ of Coxeter generators but having $M'$ as Coxeter matrix. Then all reduced expressions of the element $w\in W$ are also reduced expressions as expressions in $W'$ and are reduced expressions of a unique element, say $w'\in W'$. By the Subword Property (Theorem \ref{subword}),
 the intervals $[e,w]\subseteq W$ and $[e,w']\subseteq W'$ are isomorphic as posets and hence the special matchings of $w$ and those of $w'$ correspond. This is the reason why, in the study of the special matchings of $w$, it would be natural to assume that, for all $r,r'\in S$, the Coxeter matrix entry $m_{r,r'}$ is equal to  $\ell(w_0(r,r'))$, i.e., $w_0(r,r')$ is the top element of the parabolic subgroup $W_{\{r,r'\}}$  generated by $r$ and $r'$. In particular, this assumption would  lighten some technicalities and
 assure that $\lambda_r$,   $\lambda_{r'}$, $\rho_r$, $\rho_{r'}$ are all special matchings of $w_0(r,r')$ since $\{r,r'\}= D_L(w_0(r,r')) = D_R(w_0(r,r'))$, while there would not be loss of generality for this section.

However we are not making this assumption because it could be misleading for the applications of the results in this paper. In particular, being interested in the special matchings as possible tools to compute the parabolic Kazhdan--Lusztig polynomials associated with a subset $H\subseteq S$, it is worth noting that changing
 the Coxeter matrix as above would have the effect of changing the minimal coset representatives set $W^H$, and hence, evidently, the parabolic Kazhdan--Lusztig polynomials. 
 \end{obs}

Let $(W,S)$ be a Coxeter system. For an element $s\in S$, we let 
\[
 C_s:=\{r\in S:\, rs=sr\}.
\]

We need the following generalization of \cite[Lemma 5.4]{BCM1}.
\begin{lem}\label{5.4gen}
 Let $w\in W$ and $M$ be a special matching of $w$ with $M(e)=s$. Let $n\in \mathbb N$ and $r,c_1,\ldots,c_n,l\in S\cap [e,w]$ be such that $M(r)=sr\neq rs$, $c_1,\ldots,c_n\in C_s$ and $M(l)=ls\neq sl$. If $u=rc_1\cdots c_n l\in W$,  $\ell(u)=n+2$,  is such that ${D}_L(u)=\{r\}$, then $u\not \leq w$. 
\end{lem}
\begin{proof}
 We proceed by contradiction and take an element $u=rc_1\cdots c_n l\leq w$ of shortest length satisfying the conditions of the statement. 
Lemma 5.4 of \cite{BCM1} says that in the hypothesis of this lemma $rsl\not\leq w$ and if in addition $rl\neq lr$ then $rl\not \leq w$. Therefore  we can exclude  $n=0$   (the hypothesis ${D}_L(u)=\{r\}$ implies $rl\neq lr$ in this case) and  that $s\leq u$. Moreover,  by the minimality of $\ell(u)$ we have that $l$ is the unique right descent of $u$.
 Now we have
 \[
  M(c_1\cdots c_nl)=c_1\cdots c_nls\rhd c_1\cdots c_nl
 \]
by Lemma \ref{diheequa} (used with $M'=\rho_s$); similarly we have $M(rc_1\cdots c_n)=src_1\cdots c_n\rhd rc_1\cdots c_n$ by Lemma \ref{diheequa} (used with $M'=\lambda_s$). Therefore we have that $M(u)$ must cover the three elements $u=rc_1\cdots c_n l$, $c_1\cdots c_nls$ and $src_1\cdots c_n$. In particular, we must have that both $ls, sr\leq M(u)$. Now $M(u)$ can be obtained by adding a letter $s$ to some reduced expression of $u$. As $r$ and $l$ are respectively the only  left and right descents of $u$, all the reduced expressions of $u$ are obtained by performing braid relations in the subword $c_1\cdots c_n$. Therefore, after a possible relabelling of the letters $c_1,\ldots,c_n$  we have that $M(u)$ has a reduced expression obtained by inserting a letter $s$ in the reduced expression $rc_1\cdots c_nl$ of $u$, and so there are only three possibilities: $u_1=src_1\cdots c_nl$, $u_2=rsc_1\cdots c_nl$ and $u_3=rc_1\cdots c_nls$. But all of these must be excluded since none of them is greater than or equal 
to both $ls$ and $sr$. 
\end{proof}


The next results establish an important restriction on the elements $w$ admitting a special matching $M$  that  does not restrict to a multiplication matching of $w_0(M(e),t)$ for some $t\in S$.
\begin{lem}\label{stct}Let $w\in W$ and $M$ be a special matching of $w$. Let ${c},s,t\in S$, $m_{c,s}=2$, $m_{c,t}\geq 3$, $m_{s,t}\geq 4$. Assume $tst\leq w$ and $M(e)=s$, $M(t)=ts$ and $M(st)=tst$. Then $st{c}st\not \leq w$. If in addition $m_{{c},t}>3$ then $st{c}t\not \leq w$.
\end{lem}
\begin{proof}
 Assume that $st{c}t\leq w$. We consider the following coatoms of $st{c}t$: $x=s{c}t$, $y=st{c}$ and $z=t{c}t$ and we compute their images under $M$. 
 
 Consider the two elements $st$ and ${c}t$ coatoms of $x$. We have $M(st)=tst$ by hypothesis and $M({c}t)={c}ts$ by Lemma \ref{diheequa}.
 Therefore $M(x)$ is the only element covering the three distinct elements $x=s{c}t$, $tst$ and ${c}ts$ (Proposition \ref{k32}). We deduce that $M(x)={c}tst$.
 
 Now consider the two elements $st$ and $t{c}$, coatoms of $y$. We have $M(st)=tst$ and $M(t{c})=t{c}s$ by Lemma \ref{diheequa}. Therefore $M(y)$ is the only element covering $y=st{c}$, $tst$ and $t{c}s$  (Proposition \ref{k32}). We deduce that $M(y)=tst{c}$.
 
 We also have $M(z)=t{c}ts$ by Lemma \ref{diheequa}. 
 
 So $M(st{c}t)$ covers $st{c}t$, ${c}tst$, $tst{c}$ and $t{c}ts$. As ${c}tst$ and $tst{c}$ have only one reduced expression we deduce that necessarily $M(st{c}t)={c}tst{c}$. But this element does not cover $st{c}t$ if $m_{{c},t}>3$ and so the proof is complete in this case.
 
 %
 %
 %

 Assume now that  $m_{{c},t}=3$ and, by contradiction, that $st{c}st\leq w$. It follows from the previous discussion that $M(st{c}t)={c}tst{c}$ and, since $stct\lhd stcst$, we deduce that $M(stcst)$  must be an element covering both ${c}tst{c}$ and $st{c}st$; but a direct check shows that there is no such element.

\end{proof}

\begin{pro}\label{stcst}
 Let $w\in W$ and $M$ be a special matching of $w$. Let ${c},s,t\in S$, $m_{s,c}=2$, $m_{t,c}\geq 3$, $m_{s,t}\geq 4$. Assume $M(e)=s$, $M(t)=ts$ and $M\not \equiv \rho_s$ on $[e,w_0(s,t)]$. Then 
 \begin{enumerate}
  \item if $m_{t,{c}}>3$ then $st{c}t\not \leq w$;
  \item if $m_{t,{c}}=3$ then $st{c}st\not \leq w$.
 \end{enumerate}

\end{pro}
\begin{proof}
Clearly we may assume that $c\leq w$ otherwise we are done.

 Let $x_0$ be the minimal element $\leq w_0(s,t)$ such that $M(x_0)\neq x_0s$. If $x_0=st$ the result follows from  Lemma~\ref{stct}. Otherwise we know by Lemma~6.2 of \cite{BCM1} that the only elements $u\leq w$ covering $x_0$ and such that ${c}\leq u$ are $x_0{c}$ and ${c}x_0$. We fix a reduced expression for $w$. The result will follow if we show that the word $s\textrm{-}t\textrm{-}c\textrm{-}t$ is not a subword of the fixed reduced expression of $w$. In fact, the element $st{c}t$ has only one reduced expression if $m_{t,{c}}>3$, while $st{c}st$ has two reduced expressions, i.e. $s\textrm{-}t\textrm{-}{c}\textrm{-}s\textrm{-}t$ and $s\textrm{-}t\textrm{-}s\textrm{-}{c}\textrm{-}t$ but both of them show the subword $s\textrm{-}t\textrm{-}{c}\textrm{-}t$. Consider a subword $\cdots \textrm{-}t\textrm{-}s\textrm{-}t$ of the fixed reduced expression for $w$  which is a reduced expression for $x_0$ (observe that the unique reduced expression for $x_0$ ends with $t$ by the minimality property of $x_0$). Now 
consider any possible occurrence of the subword $s\textrm{-}t\textrm{-}{c}\textrm{-}t$. By considering all possible positions of the letter ${c}$ with respect to the fixed reduced expression for $x_0$, one can show that we always contradict Lemma 6.2 of \cite {BCM1}. For example, if the letter $c$ appears to the right of the chosen reduced expression for $x_0$ we have that the element $x_0tct$ covers $x_0$ but is distinct from $x_0c$ and $cx_0$; similar arguments apply in the other cases.
\end{proof}
\begin{pro}\label{R4-R5}
 Let $w\in W$, $s,t\in S$, $s,t\leq w$, and $M$ be a special matching of $w$ such that $M(e)=s$, $M(t)=ts$ and $J=\{r\in S:\, M(r)=sr\}$. Then
 \begin{enumerate}
  \item if $\alpha\in \{s,t\}$ is such that $\alpha\leq (w^J)^{\{s,t\}}$, then $M(\alpha u)=\alpha M(u)$ for all $u\leq w_0(s,t)$;
  \item if $s\leq \,^{\{s\}}\!(w_J)$, then $M(us)=M(u)s$ for all $u\leq w_0(s,t)$.
 \end{enumerate}

\end{pro}
\begin{proof}
 We can clearly assume that $M\not\equiv \rho_s$ on $[e,w_0(s,t)]$. By Theorem \ref{ma16} we have that 
\[M(u)=(u^{J})^{\{s,t\}}\, \cdot \, M\Big ((u^{J})_{\{s,t\}} \, \cdot \, _{\{s\}} (u_{J})\Big )\, \cdot \,
 ^{\{s\}}(u_{J}),\]
for all $u \leq w$. 
By  Proposition \ref{734}, (2), $s$ and $t$ cannot be both $\leq (w^J)^{\{s,t\}}$. This  implies that $[e,w_0(s,t)]$ has at most 4 elements more than 
$[e,(w^{J})_{\{s,t\}} \,  \cdot  \,  _{\{s\}} (w_{J})]$, since $\ell(w_0(s,t))$ can be at most 
$2+ \ell((w^{J})_{\{s,t\}} \,  \cdot  \,  _{\{s\}} (w_{J}))$.
(It may happen that  $\ell(w_0(s,t)) = 2+ \ell((w^{J})_{\{s,t\}} \,  \cdot  \,  _{\{s\}} (w_{J}))$ only if $s \leq  \, ^{\{s\}}(w_{J})$ and there exists  $\alpha \in \{s,t\}$  such that $\alpha \leq (w^J)^{\{s,t\}}$; the converse is not necessarily true.)

Let us prove (1). The condition $\alpha \leq (w^J)^{\{s,t\}}$ implies that $\alpha \in D_L(w_0(s,t))$, which means that  $\lambda_{\alpha}$ is a special matching of $w_0(s,t)$. Since $\alpha \notin D_R ( (w^J)^{\{s,t\}} )$, there must be 
$r\in S \setminus \{s,t\}$, $r$ not commuting with $\alpha$, such that $\alpha r \leq (w^J)^{\{s,t\}}$. 
By contradiction, let $x $  be minimal among the elements in $ [e,w_0(s,t)] $ such that $M(\alpha x)  \neq \alpha M(x)$. By the minimality of $x$ we have that $x$ is the minimal element of an orbit 
of $\langle M, \lambda_{\alpha} \rangle $ of cardinality at least 6 and so  $\ell (x)< \ell(w_0(s,t)) -2$, $x \leq (w^{J})_{\{s,t\}} \,  \cdot  \,  _{\{s\}} (w_{J})$ and so,  by the Subword Property (Theorem \ref{subword}),
 $\alpha r x \leq w$. (Recall that $\langle M, \lambda_{\alpha} \rangle $ denotes the group generated by $M$ and $\lambda_{\alpha}$.)  Moreover, the minimality of $x$ ensures that  $\alpha x\rhd x$ and so $\alpha rx\rhd\alpha x\rhd x$.
Now we remark that $M(\alpha rx)=\alpha r M(x)$: this follows by observing that $((\alpha r x)^J)^{\{s,t\}}=\alpha r$ and that $^{\{s\}}((\alpha rx)_J)=e$, and using Theorem \ref{ma16}, (i). In fact, if $r\notin J$ this is clear and if $r\in J$ we have $rs=sr$ by Proposition \ref{734}, (1) and so $\alpha=t$, $x=st\cdots$ and the result again follows easily. 
By the definition of a special matching, 
$M(\alpha x)\lhd M(\alpha r x) = \alpha r M(x)$, but this happens if and only if   $M(\alpha x) = \alpha M(x)$. This is a contradiction.

Let us  prove  (2).   Note that, since $t \not \leq \, ^{\{s\}} (w_{J})$, we have that $s \in D_R ( w_0(s,t))$, that is, $\rho_s$ is a special matching of $w_0(s,t)$.

By contradiction, let   $x $  be minimal among the elements in $ [e,w_0(s,t)] $ such that $M( xs)  \neq  M(x)s$. Hence $x$ is the minimal element of an orbit of $\langle M, \rho_{s} \rangle$ of cardinality at least 6: thus $\ell (x)< \ell(w_0(s,t)) -2$ and $x \leq (w^{J})_{\{s,t\}} \,  \cdot  \,  _{\{s\}} (w_{J})$.  
The hypothesis $s\leq \, ^{\{s\}} (w_{J})$ implies that there exists a generator $p\in J$ 
not commuting with $s$ such that   $ps\leq \, ^{\{s\}} (w_{J})$. By the Subword Property,  we   have  $xps \leq w$.

 Since  $x$ 
is the minimal element of an orbit 
of $\langle M, \rho_{s} \rangle $, we have $x \lhd \rho_{s} (x) = xs$ (i.e., $\ell(xs)=\ell(x)+1$) and hence   $xs \lhd xps$,  since  $\ell(xs) = \ell(x p s)-1$.
Now observe that $M(xps)=M(x)ps$. By the definition of a special matching, 
$M(xs)$ should be $\leq M(xps) =  M(x) p s $, but this happens if and only if   $M(xs) = M(x)s$. This is a contradiction.
\end{proof}

\section{Left and right systems}\label{leftright}
Propositions \ref{734} and \ref{R4-R5} lead us to introduce the following definition, which has a right version and, symmetrically, a left version. 

\begin{defi}
A {\em  right system for $w$} is a quadruple $\mathcal R=(J,s,t,M_{st})$ such that:

\begin{enumerate}

\item[R1.] $J\subseteq S$, $s\in J$, $t\in S\setminus J$, and $M_{st}$ is a special matching of $w_0(s,t)$ such that  $M_{st}(e)=s$ and  $M_{st}(t)=ts$;

\item[R2.] $(u^{J})^{\{s,t\}}\, \cdot \, M_{st} \Big ((u^{J})_{\{s,t\}} \, \cdot \, _{\{s\}} (u_{J})\Big )\, \cdot \,
 ^{\{s\}}(u_{J}) \leq w$,  for all $u\leq w$;

 \item[R3.] 
if $r\in J$ and $r \leq w^J$, then $r$ and $s$ commute;

\item[R4.]
\label{ddddxxxx}
if $\alpha\in \{s,t\}$ is such that $\alpha \leq (w^J)^{\{s,t\}}$, then $M_{st}$ commutes with $\lambda_\alpha$ on $[e,w_0(s,t)]$;

\item[R5.] if $s\leq \, ^{\{s\}} (w_{J})$, then $M_{st}$ commutes with $\rho_s$ on $[e,w_0(s,t)]$.
\label{pure}
\end{enumerate} 
\end{defi}
Some additional properties of right systems can be immediately deduced from their definition. 
\begin{lem}\label{stleq}
Let  $\mathcal R=(J,s,t,M_{st})$ be a right system for $w$ such that $s,t\leq (w^J)^{\{s,t\}}$. Then 
\[
 M_{st}(x)=xs
\]
for all $x\leq w_0(s,t)$.
\end{lem}
\begin{proof}
 It is enough to show that if $x\leq w_0(s,t)$ is such that $xs\rhd x$ then $M(x)=xs$. To show this we proceed by induction on $\ell(x)$, the result being clear if $\ell(x)=0$. So let $\ell(x)>0$, $\alpha\in D_L(x)$ and $y=\alpha x\lhd x$. Then by induction hypothesis we have $M(y)=ys$ and so $M(x)\rhd x$ by the definition of special matchings. Moreover, we have $M(\alpha x)=\alpha M(x)$ by Property R4 and since $M(\alpha x)=M(y)=ys=\alpha xs$ the result follows.  
\end{proof}
\begin{lem}\label{davaw}
 Let $\mathcal R=(J,s,t,M_{st})$ be a right system for $w$, and $v\leq w$. Then $s\leq \, ^{\{s\}} (v_{J})$ implies $s\leq \, ^{\{s\}} (w_{J})$.
\end{lem}
\begin{proof}
 If $s\leq \, ^{\{s\}} (v_{J})$ then there exists $r\in J$ such that $rs\neq sr$ and $rs\leq \, ^{\{s\}} (v_{J})\leq w$. Therefore $r\textrm{-}s$ must be a subword of every reduced expression of $w$. Since $r\not \leq w^J$ by Property R3, this implies that $rs\leq w_J$ and in particular we necessarily have $s\leq \, ^{\{s\}} (w_{J})$, by Proposition \ref{mantiene}.
\end{proof}

\begin{defi}
 A {\em left system for $w$} is a right system for $w^{-1}$. It is an immediate verification that the datum of a left system is equivalent to the datum of a quadruple  $\mathcal L=(J,s,t,M_{st})$ such that:

\begin{enumerate}

\item[L1.] $J\subseteq S$, $s\in J$, $t\in S\setminus J$, and $M_{st}$  is a special matching of $w_0(s,t)$ such that $M_{st}(e)=s$ and  $M_{st}(t)=st$;

\item[L2.] $    (_Ju)^{\{s\}} \, \cdot \, M_{st} \Big( \, (_Ju)_{\{s\}}  \,  \cdot  \,  _{\{s,t\}} (^J  u) \Big) \, 
\cdot \, ^{\{s,t\}}(^J u)    \leq w$,  for all $u\leq w$;

 \item[L3.] 
if $r\in J$ and $r \leq \,^Jw$, then $r$ and $s$ commute;

\item[L4.]
if $\alpha\in \{s,t\}$ is such that $\alpha \leq \,^{\{s,t\}}(\,^Jw)$, then $M_{st}$ commutes with $\rho_\alpha$ on $[e,w_0(s,t)]$;

\item[L5.] if $s\leq \,  (_{J}w)^{\{s\}}$, then $M_{st}$ commutes with $\lambda_s$ on $[e,w_0(s,t)]$.
\label{pure}
\end{enumerate} 

\end{defi}

%

With a right system $\mathcal R=(J,s,t,M_{st})$ for $w$, we associate a map $M_{\mathcal R}$ on $[e,w]$ in the following way. 
Given $u\leq w$, 
we set $$M_\mathcal R (u) = (u^{J})^{\{s,t\}}\, \cdot \, M_{st} \Big( (u^{J})_{\{s,t\}} \,  \cdot  \,  _{\{s\}} (u_{J}) \Big) \, 
\cdot \, ^{\{s\}}(u_{J}).$$
Symmetrically, we associate  with any left system $\mathcal L$ for $w$ a map $_\mathcal L M$ on $[e,w]$ by setting 
\[
 _{\mathcal L}M(u)=\big(M_\mathcal L(u^{-1})\big)^{-1},
\]
where $M_\mathcal L$ is the map on $[e,w^{-1}]$ associated to $\mathcal L$ as a right system for $w^{-1}$.

It is not clear at all that such $M_\mathcal R$ (or, equivalently, $_{\mathcal L}M$) defines a matching of $w$. For this we need to show that $\{u,M_\mathcal R(u)\}$ is always an edge in the Hasse diagram and that $M_\mathcal R$ is an involution.

\begin{rem}\label{3.3}Note that, if $s\in D_R(w)$, $t\in S\setminus\{s\}$,  $J= \{s\}$ and $M_{st}= \rho_s$, we obtain a right system with associated  right multiplication matching ($M = \rho_s$ on the entire interval $[e,w]$). Symmetrically,  we obtain left multiplication matchings as special cases of matchings associated with  left systems.
\end{rem}

Evidently, distinct systems for $w$ might give rise to the same maps on $[e,w]$.

In order to show that the map $M_\mathcal R$ associated with a right system is a matching we need the following elementary results. 
\begin{lem}
\label{lemma0cap}
Fix $H \subseteq S$ and $u=u^H\cdot u_H\in W$. Let  $j \in D_R(u) \setminus H$. Then  $j \in D_R(u^H)$.
\end{lem}
 \begin{proof}
We proceed by induction on $\ell(u_H)$. If $\ell(u_H)=0$, the assertion is clear. So assume $\ell(u_H)>0$, and let $h \in D_R(u_H)$.  It follows immediately that $h\in D_R(u)$ and from 
the ``Lifting Property'' (see, e.g., \cite[Proposition 2.2.7]{BB} or \cite[Proposition 5.9]{Hum}), we also have $j \in D_R(uh)$. Since $uh=u^H \cdot (uh)$, with $uh\in W_H$, by induction hypothesis  $j \in D_R(u^H)$.
\end{proof}

\begin{lem}
\label{lemma00cap}
Fix $u\in W$ and $t,j \in S$, with $t\leq u$ and $j \in D_R(u)$.  Assume that there exists a reduced expression  $X$ for $u$ such that $t \textrm{-}j$ is not a subword of $X$. Then $t$ and $j$ commute.
\end{lem}
 \begin{proof}
Since $j \in D_R(u)$, there exists a reduced expression $X'$ for $u$ ending with the letter $j$ and hence having $t \textrm{-}j$ as a  subword. By the Subword Property, $tj \leq u$ but, since $t \textrm{-}j$ is not a subword of $X$, $tj$ must be equal to $jt$. 
\end{proof}

\begin{lem}
\label{lemma2cap}
Let $u\in W$, $s,t \in S$ with $m_{s,t}\geq 3$, $s\nleq u^{\{s,t\}}$ and $st\leq u_{\{s,t\}}$. Let $j\in  D_{R}(u)\setminus \{s,t\}$. 
Then  $j$ commutes with $s$ and $t$.
\end{lem}
\begin{proof}
 Consider two reduced expressions $X$ and $Y$  for $u$ with the following properties: $X$ ends with the letter $j$  (such a reduced expression exists since $j\in  D_{R}(u)$) and $Y$ is the concatenation of a reduced expression for $u^{\{s,t\}}$ and a reduced expression for $u_{\{s,t\}}$.

Since $st \leq u$ and $st\neq ts$, the expression $X$ has $s \textrm{-}t$ as a subword, and then also $s \textrm{-}t \textrm{-}j$; this implies $x:=stj \leq  u$. Now if we let Red$(x)$ be the set of all reduced expressions for $x$ we have
\[
 \textrm{Red}(x)=\begin{cases} \{s\textrm{-}t\textrm{-}j\} & \textrm{if }jt\neq tj,\\ \{s\textrm{-}t\textrm{-}j,s\textrm{-}j\textrm{-}t\} & \textrm{if }jt=tj \textrm{ and }sj\neq js,\\ \{s\textrm{-}t\textrm{-}j,s\textrm{-}j\textrm{-}t,j\textrm{-}s\textrm{-}t\} & \textrm{if }jt=tj \textrm{ and }sj=js.
 \end{cases}
\]
Since $s\nleq u^{\{s,t\}}$, we have that $s\textrm{-}t\textrm{-}j$ and $s\textrm{-}j\textrm{-}t$ can not be subexpressions of $Y$ and therefore we conclude that $jt=tj$ and $sj=js$.
\end{proof}

\begin{pro}
\label{definiscematching}
Let $\mathcal R=(J,s,t,M_{st})$ be a right system for $w$ and $u\leq w$.  
Then 
\begin{itemize}
\item  $(M_\mathcal R(u)^{J})^{\{s,t\}}=  (u^{J})^{\{s,t\}}$,
\item  $(M_\mathcal R(u)^{J})_{\{s,t\}} \,  \cdot  \,  _{\{s\}} (M_\mathcal R(u)_{J}) = M_{st} \Big( (u^{J})_{\{s,t\}} \,  \cdot  \,  _{\{s\}} (u_{J}) \Big) $,
\item $  ^{\{s\}}(M_\mathcal R(u)_{J}) = \, ^{\{s\}}(u_{J})$.
\end{itemize}
\end{pro}
 \begin{proof}
 For notational convenience, we let  $M=M_\mathcal R$ and for all $u\leq w$ we let ${\bar u}=(u^{J})_{\{s,t\}} \,  \cdot  \,  _{\{s\}} (u_{J})\leq w_0(s,t)$.
All the assertions follow at once from the following claim:
\begin{eqnarray}
\label{capraia}
 (u^{J})^{\{s,t\}}\, \cdot \, M_{st} ( {\bar u} ) & \in & W^J \cup ( W^J \cdot s ).
\end{eqnarray}
In fact, assume that $(u^{J})^{\{s,t\}}\, \cdot \, M_{st} ( {\bar u} )\varepsilon \in W^J$, with $\varepsilon\in \{e,s\}$. Then $M(u)^J=(u^{J})^{\{s,t\}}\, \cdot \, M_{st} ( {\bar u} )\varepsilon$ and $M(u)_J=\varepsilon\cdot \, ^{\{s\}}(u_{J})$. From this parabolic decomposition it follows that $(M(u)^J)^{\{s,t\}}= (u^{J})^{\{s,t\}}$, $(M(u)^J)_{\{s,t\}}=M_{st} ( {\bar u} )\varepsilon$, $ _{\{s\}} (M(u)_{J})=\varepsilon$ and  $^{\{s\}}(M(u)_{J}) = \, ^{\{s\}}(u_{J})$, and the three assertions are proved.

Let us now prove (\ref{capraia}).

 By definition, $ (u^{J})^{\{s,t\}}\, \cdot \,  {\bar u} =u^J \cdot  \,  _{\{s\}} (u_{J})\in W^J \cup ( W^J \cdot s )$, hence  (\ref{capraia}) holds when $ M_{st} ( {\bar u}) =  \rho_s ( {\bar u}) $. This happens if
 ${\bar u}  \in \{e,s,t,ts\}$, by Property R1, and if  both $s$ and $t$ are  $\leq (u^J)^{\{s,t\}}$, by Lemma \ref{stleq} (recall that $s,t\leq (u^J)^{\{s,t\}}$ implies  $s,t\leq (w^J)^{\{s,t\}}$, by Proposition \ref{mantiene}).

In the remaining cases, we will prove  (\ref{capraia}) by showing that, if $j \in J\setminus \{s\}$ is a right descent of $ (u^{J})^{\{s,t\}}\, \cdot \,M_{st} ( {\bar u})\varepsilon$, for some $\varepsilon \in \{e,s\}$ then   
\begin{enumerate}
\item $j\in D_R((u^J)^{\{s,t\}}) $,
\item $j$ commutes with $s$ and $t$,
\end{enumerate}
which is in contradiction with $ (u^{J})^{\{s,t\}}\, \cdot \, {\bar u}  \in W^J \cup ( W^J \cdot s )$.

So assume that  $j \in J\setminus \{s\}$ is a right descent of $ (u^{J})^{\{s,t\}}\, \cdot \, M_{st} ( {\bar u})\varepsilon$. Lemma \ref{lemma0cap}, with $H=\{s,t\}$, implies that (1) holds, and by Property R3 we have that $j$ commutes with $s$. We need to show that $j$ commutes with $t$.

If $t \not \leq (u^J)^{\{s,t\}}$, we may conclude using Lemma \ref{lemma00cap} applied to $ (u^{J})^{\{s,t\}}\, \cdot \,M_{st} ( {\bar u})\varepsilon$. If 
$s \not \leq (u^J)^{\{s,t\}}$, we may conclude using Lemma \ref{lemma2cap} (since $\bar u \notin \{e,s,t,ts\}$, $s$ does not commute with $t$).
\end{proof}
\begin{cor}\label{definiscematching2}
 Let $\mathcal R=(J,s,t,M_{st})$ be a right system for $w$. Then $M_\mathcal R$ is a matching of $[e,w]$. Moreover, for all $u\leq w$, we have $u\lhd M_\mathcal R(u)$ if and only if $\bar u\lhd M_{st}(\bar u)$, where $\bar u=(u^{J})_{\{s,t\}} \,  \cdot  \,  _{\{s\}} (u_{J})$.
\end{cor}
\begin{proof}
Proposition \ref{definiscematching} immediately implies that $M_\mathcal R$ is an involution. To show that $\{u,M_\mathcal R(u)\}$ is an edge in the Hasse diagram of $[e,w]$ with no lack of generality we can assume that $\ell(u)<\ell(M_\mathcal R(u))$. This can happen only if $\bar u\lhd M_{st}(\bar u)$ and so we have that $M_{\mathcal R}(u)$ has an expression obtained by adding one letter to a reduced expression of $u$. As $\ell(u)<\ell(M_\mathcal R(u))$ this forces such an expression to be reduced and $u\lhd M_\mathcal R(u)$. The last assertion is an immediate consequence.
\end{proof}

In the rest of the paper, we will often use Corollary \ref{definiscematching2} without explicit mention.

We conclude this section with the crucial observation that any special matching of $w$ is the map associated to a right or left system. 
\begin{rem}
\label{oss}
Given a matching $M$ of $w$, we can define a matching $ \tilde{M} $ of $ w^{-1} $ by setting $ \tilde{M}(x)=  (M(x^{-1}))^{-1} $,
for all $ x\leq w^{-1} $. It satisfies the following properties:
\begin{enumerate}
\item $M$ is special if and only if  $ \tilde{M}$ is special;
\item  $M$ is associated with a right system if and only if $ \tilde{M} $ is associated with a left system;
\item $M(y)=ys$ if and only if  $\tilde{M}(y^{-1})=sy^{-1} $, $s\in S$;
\item  $\tilde{\tilde{M}} = M$.
\end{enumerate}
\end{rem}

\begin{thm}
\label{caratteri}
Let $w$ be any element of any arbitrary Coxeter group $ W$ and let $M$ be a special matching of $w$. Then $M$ is  associated with a right or a left system of $w$.  
\end{thm}

\begin{proof}

We first assume we are in case (ii) of Theorem \ref{ma16}. If $w^J=e$, then  $M=\lambda_s$ and hence $M$ is associated with a left system, as we noted in Remark \ref{3.3}. If $w^J\neq e$, then  $M$  is associated with the right system $(J,s,t,M_{st})$, where  $J$ and $s$ are those of Theorem \ref{ma16}, $t$ is any Coxeter generator
 in $S\setminus J$ among those that are $\leq w^J$, and $M_{st}$ is the  restriction of $M$ on $[e,w_0(s,t)]$,  which is the right multiplication by
 $s$ (i.e., $M_{st}=\rho_s$). The proof is straightforward (the unique non immediate property, which is Property R3, follows from Proposition \ref{734}, (1)).

Now assume we are not in case (ii) of Theorem \ref{ma16}. By Remark \ref{oss}, changing $M$ with $\tilde{M}$ if necessary, we can assume that we are in case (i) of Theorem \ref{ma16}.  Take as $J$, $s$ and $t$ those of Theorem \ref{ma16}, 
and as $M_{st}$ the restriction of $M$ to  
 $[e,w_0(s,t)]$. We prove that  $(J,s,t,M_{st})$ is a right system. Properties R1 and R2  are immediate, Property R3 follows from Proposition \ref{734}, (1), and Properties R4 and R5 are the contents of Proposition \ref{R4-R5}.   
\end{proof}

\section{Classification of special matchings}

This section is devoted to the proof of our main result which is a characterization and classification of all special matchings of an arbitrary element $w$ in an arbitrary Coxeter system $(W,S)$. In particular, we prove that the matching $M_{\mathcal R}$ associated to a right system $\mathcal R$ is always a special matching and we conclude by characterizing those (right or left) systems that give rise to the same special matching of $w$.

For $s\in {I}\subseteq S$, we let 
\[
 K_s({I}):= C_s\cup (S\setminus {I}).
\]
For short, we write $K(I)$ instead of $K_s(I)$ when no confusion arises.

We first concentrate on the simpler and enlightening special case of a right system $\mathcal R=(I,s,t,\rho_s)$, i.e. a right system where the special matching $M_{st}$ of $w_0(s,t)$ is given by the right multiplication by $s$. Note that in this case Properties R1, R4 and R5 are automatically satisfied and that the matching $M_\mathcal R$ is given by 
\[
 M_\mathcal R(u)=u^Isu_I
\]
for all $u\leq w$ by definition of $M_\mathcal R$, and so the element $t$ does not play any role.

\begin{thm}
\label{fa1}
 Let $w\in W$ and $s \in I\subseteq S$ be such that $u^{I}su_{I} \leq w$ for all $u\leq w$. Then the map $M$ on $[e,w]$ given by $M(u)=u^{I}su_{I}$ for all $u\leq w$  is a special matching of $w$ if and only if $w^{I}\in W_{K({I})}$. In particular, for all $t\notin I$, we have that $\mathcal R=(I,s,t,\rho_s)$ is a right system if and only if $M_\mathcal R$ is a special matching.
\end{thm}
\begin{proof}
Let ${I}'={I}\setminus C_s$ and $K'=K({I})\setminus C_s$, so that $S={I}'\cup K'\cup C_s$, the union being disjoint.
  Suppose first that $M_{}$ is a special matching of $[e,w]$. If, by contradiction,  $w^{I} \notin W_{K({I})}$ there exists $r\in {I}'$  such that $r\leq w^{I}$. If $\ell(w^{I})=\ell$, let 
  \[
   i=\max\{j:\, w^{I}=s_1\cdots s_{\ell} \textrm{ for some }\, s_1,\ldots,s_{\ell}\in S \textrm{ and } s_j\in {I}' \}
  \]
  and fix a reduced expression $w^{I}=s_1\cdots s_\ell$ such that $s_i\in {I}'$.
  
 Now we claim that for all $j>i$ we have $s_j\notin K'$. Otherwise take a minimal such $j$ and consider the element $u=s_is_{i+1}\cdots s_j$. We have that $s_i$ is its only left descent by the maximality of $i$. We also have $s_{i+1},\cdots, s_{j-1} \in C_s$, by the maximality of $i$ and the minimality of $j$,  contradicting Lemma \ref{5.4gen}. 

 Therefore we have $s_{i+1},\ldots,s_{\ell}\in C_s$. 
 Now we have $M_{}(s_i\cdots s_\ell)=ss_i\cdots s_\ell$ by Lemma \ref{diheequa} and $M_{}(s_i\cdots s_{\ell})=s_i\cdots s_\ell s$ by the definition of $M_{}$ since $s_i\cdots s_{\ell}\in W^{I}$. Since $s_{i+1},\ldots,s_{\ell}\in C_s$ this implies $s_is=ss_i$ contradicting the fact that $s_i\in {I}'$.
 
 Suppose now that $w^{I}\in W_{K({I})}$. Let $u\lhd v\leq w$ be such that $u\lhd M_{}(u)\neq v$. By Proposition~\ref{nonconfigurazione}, we have to show that $v\lhd M(v)$. As $u^{I}\leq v^{I}\leq w^{I}$ (Proposition~\ref{mantiene}), we clearly have $u^{I},v^{I}\in W_{K({I})}$. We know that a reduced expression for $u$ can be obtained from any reduced expression of $v$ by deleting one letter. If we consider a reduced expression for $v$ given by the concatenation of a reduced expression of $v^{I}$ with a reduced expression of $v_{I}$, we have two cases to consider according to whether such letter comes from $v^{I}$ or $v_{I}$.
 
 (1) There exists $a\lhd v_{I}$ such that $u=v^{I}a$. In this case we have $u^{I}=v^{I}$ and $u_{I}=a$ and so $M_{}(u)=u^{I}sa=v^{I}sa$, with $sa\rhd a$. As $a\lhd v_{I}$, $sa\rhd a$ and $sa\neq v_{I}$ (since otherwise $M_{}(u)=v$), we have $sv_{I}\rhd v_{I}$ by the Lifting Property: this implies $M_{}(v)\rhd v$.
 
 (2) There exists $a\lhd v^{I}$ such that $u=av_{I}$. As $a\lhd v^{I}$ we have $a\in W_{K({I})}$ by hypothesis. Therefore $a_{I}\in W_{K({I})}\cap W_{I}$ and in particular $a_{I}$ commutes with $s$. Moreover we have $u=a^{I}a_{I}v_{I}$ and so $u^{I}=a^{I}$ and $u_{I}=a_{I}v_{I}$ and hence
 \[
  M_{}(u)=u^{I}su_{I}=a^{I}sa_{I}v_{I}=a^{I}a_{I}sv_{I}=asv_{I}.
 \]
 It follows that $\ell(sv_{I})>\ell(v_{I})$ since otherwise
 \[
  \ell(M_{}(u))\leq \ell(a)+\ell(sv_{I})\leq \ell(a)+\ell(v_{I})=\ell(u)
 \]
 contradicting the assumption $M_{}(u)\rhd u$. 
 
 The relation $M_{}(v)\rhd v$ immediately follows.

The last statement is now straightforward.
\end{proof}

\begin{cor}
Fix $w\in W$. 
\begin{enumerate}
\item 
\label{uno}Let $\mathcal R=(I,s,t,\rho_s)$ be a right system for $w$. Then $\tilde{\mathcal  R}=(J,s,t,\rho_s)$, where $J={I}\cup C_s\setminus\{t\}$, is also a right system for $w$ and $M_\mathcal R=M_{\tilde {\mathcal R}}$.

\item 
\label{due}
Let $\mathcal R =(I,s,t,\rho_s)$ and $\mathcal R'=(I',s,t',\rho_s)$ be right systems for $w$. Then $M_\mathcal R=M_{\mathcal R'}$ if and only if $I\cup C_s=I'\cup C_s$. 
\end{enumerate} 
\end{cor}
\begin{proof}
We first prove (\ref{uno}). Observe that we have $K({I})=K(J)$. By Theorem \ref{fa1} we have $u^{I}\in W_{K({I})}$ for all $u\leq w$; therefore $(u^{I})_{J}\in W_{K({I})}\cap W_{J}$ and in particular $(u^{I})_{J}$ commutes with $s$. Therefore
 \begin{align*}
  M_{\tilde{\mathcal  R}}(u)&=M_{\tilde{\mathcal  R}}((u^{I})^{J} \cdot (u^{I})_{J} \cdot u_{I})=(u^{I})^{J}\cdot s\cdot (u^{I})_{J}\cdot u_{I}\\ &=(u^{I})^{J} \cdot (u^{I})_{J}\cdot s\cdot u_{I}=u^{I}\cdot s\cdot u_{I}\\ &=M_{\mathcal R}(u)
 \end{align*}
 for all $u\leq w$.
 
We now prove (\ref{due}). If ${I}\cup C_s={I}'\cup C_s$ we have  $M_{\mathcal R}=M_{{\mathcal R}'}$  by Lemma \ref{diheequa}. 
 Suppose $M_{\mathcal R}=M_{\mathcal R'}$ and let $r\notin C_s$. Then
 \[
  r\in K({I}) \Leftrightarrow M_{\mathcal R}(r)=rs
 \]
 and similarly for ${I}'$, and the result  follows.
\end{proof}
\begin{rem}
 A word of caution is needed. If ${I},{I}'\subseteq S$ are such that $s\in {I}\cap {I}'$, $t\notin I\cap I'$ and $I\cup C_s=I'\cup C_s$ then it is not true that $\mathcal R=(I,s,t,\rho_s)$ is a right system if and only if $\mathcal R'=(I',s,t,\rho_s)$ is.  This fails, e.g., in $A_4$ with $w=s_4s_2s_3s_2s_1$, $s=s_3$, $t=s_4$, ${I}=\{s_1,s_2,s_3\}$ and ${I}'=\{s_2,s_3\}$. In this case we have $I\cup C_s=I'\cup C_s=\{s_1,s_2,s_3\}$ and so $K(I)=K(I')=\{s_1,s_3,s_4\}$. Then we can observe that  $w^I=s_4\in W_K$ and in fact one can verify that the map $u\mapsto u^Is_3u_I$ defines a special matching of $w$. On the other hand $w^{I'}=s_4s_3s_2s_1\notin W_K$ and in fact the map $u\mapsto u^{I'}s_2u_{I'}$, although it defines a matching on $[e,w]$, it is not a special matching of $w$ as for example it shows the following N-configuration (see Proposition \ref{nonconfigurazione}):
 
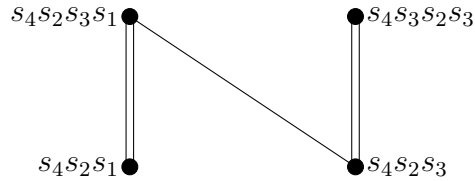
 $$
\begin{tikzpicture}

\node[left] at (-1.5,1.5){$s_4s_2s_1$};
\node[right] at (1.5,1.5){$s_4s_2s_3$};

\draw{(-1.45,3.5)--(-1.45,1.5)}; 
\draw{(-1.55,3.5)--(-1.55,1.5)}; 
\draw{(1.45,3.5)--(1.45,1.5)}; 
\draw{(1.55,3.5)--(1.55,1.5)}; 
\draw{(-1.5,3.5)--(1.5,1.5)}; 

\draw[fill=black]{(1.5,1.5) circle(3pt)};
\draw[fill=black]{(1.5,3.5) circle(3pt)};
\draw[fill=black]{(-1.5,3.5) circle(3pt)};
\draw[fill=black]{(-1.5,1.5) circle(3pt)};

\node[left] at (-1.5,3.5){$s_4s_2s_3s_1$};
\node[right] at (1.5,3.5){$s_4s_3s_2s_3$};

\end{tikzpicture}.
$$

\end{rem}
For $s\in J\subseteq S$ let $M_{J,s}(u):=u^J\,s\,u_J$.
\begin{cor}\label{charpar}Let $w\in W$. Then
\[
\{M_{J,s} \textrm{ with } s\in S,\, C_s\subseteq J\subsetneq S,\, w^J\in W_{C_s\cup (S\setminus J)},\,M_{J,s}(u)\leq w\,\,\forall u\leq w\}
 \]
 is a complete list of all distinct special matchings of $w$ associated with a right system of the form  $\mathcal R=(I,s,t,\rho_s)$. 

\end{cor}

The next result shows that it is not necessary to know the parabolic decomposition $u^Ju_J$ of an element $u\leq w$ to compute $M_{J,s}(u)$. The corresponding generalization to the general setting of right systems will be the crucial step in our classification. 

\begin{pro}Let $w\in W$, $s\in S$ and $C_s\subseteq J\subseteq S$ be such that $M_{J,s}$ is a special matching of $w$ and $K=C_s\cup (S\setminus J)$.  
Let $u\leq w$, $u=u_1u_2$ with $\ell(u)=\ell(u_1)+\ell(u_2)$, $u_1\in W_K$ and $u_2\in W_J$. Then
\[
 M_{J,s}(u)=u_1su_2.
\]
\end{pro}
\begin{proof}We proceed by induction on $\ell(u_1)$.
 If $\ell(u_1)=0$ or $u_1=u^J$ it is clear. Otherwise there exists $c\in C_s$ such that $c\in {D}_R(u_1)$. Then if we let $u_1'=u_1c$ and $u_2'=cu_2$ we have $\ell(u_1')=\ell(u_1)-1$ and $\ell(u_1')+\ell(u_2')=\ell(u)$. Moreover $u_1'\in W_K$ and $u_2'\in W_J$. Therefore by our induction hypothesis we have
 \[
  M_{J,s}(u)=M_{J,s}(u_1'u_2')=u_1'su_2'=u_1cscu_2=u_1su_2,
 \]
as desired.
\end{proof}
Now we concentrate on the case of a right system $(J,s,t,M_{st})$ such that $M_{st}\not \equiv \rho_s$. The first result that we need is the following analogue of Proposition \ref{stcst}.
\begin{lem}
 Let $(J,s,t,M_{st})$ be a right system for $w$ such that $M_{st}\not \equiv \rho_s$, and let ${c}\in S$ be such that $m_{c,s}=2$ and $m_{c,t}\geq 3$. Then
 \begin{enumerate}
  \item if $m_{{c},t}>3$ then $st{c}t\not \leq w$;
  \item if $m_{{c},t}=3$ then $st{c}st\not \leq w$.
 \end{enumerate}
\end{lem}
\begin{proof}If $m_{{c},t}>3$ let $u=st{c}t$ and if $m_{{c},t}=3$ let $u=st{c}st$. Then in both cases we have $u^J=u$ (notice that $M_{st}\not \equiv \rho_s$ implies  $m_{s,t}>3$) and
\[
 (u^J)^{\{s,t\}}=st{c}.
\]
In particular we have $s,t\leq (u^J)^{\{s,t\}}\leq (w^J)^{\{s,t\}}$ which contradicts Lemma \ref{stleq}.
\end{proof}

\begin{thm}
\label{fa2}
Let $\mathcal R=(J,s,t,M_{st})$ be a right system for $w$ such that $M_{st}\not \equiv \rho_s$ and $K:=K(J)=(S\setminus J)\cup C_s$.
 Let $u=u_1u_2u_3\leq w$ be such that the following conditions are satisfied
 \begin{itemize}
  \item $\ell(u)=\ell(u_1)+\ell(u_2)+\ell(u_3)$;
  \item $u_1 \in W_{K\setminus\{s\}}\cup W_{K\setminus\{t\}}$;
  \item $u_2 \in W_{s,t}$;
  \item $u_3 \in W_J$;
  \item $s,t\not \in {D}_R(u_1)$;
  \item $s \notin {D}_L(u_3)$.
 \end{itemize}
 Then  
 \[
  M_\mathcal R (u)=u_1 M_{st}(u_2) u_3
 \]
 and $u \lhd M_\mathcal R(u)$ if and only if $u_2 \lhd M_{st}(u_2)$.
 \end{thm}
\begin{proof}
Let 
\[
 \varepsilon=\begin{cases}
              e& \textrm{if }u_2s>u_2;\\ s&\textrm{if } u_2s<u_2.
             \end{cases}
\]
We proceed by induction on $\ell(u_1)$. If $\ell(u_1)=0$ we have $u^J=u_2\varepsilon$ and $u_J=\varepsilon u_3$, and therefore $(u^J)^{\{s,t\}}=e$, $(u^J)_{\{s,t\}}=u_2\varepsilon$, $_{\{s\}}(u_J)=\varepsilon$ and $^{\{s\}}(u_J)=u_3$. Therefore 
\[
 M_{\mathcal R}(u)=(u^J)^{\{s,t\}} \cdot M_{st}((u^J)_{\{s,t\}} \cdot \,_{\{s\}}(u_J)) \cdot \,^{\{s\}}u_J=M_{st}(u_2\varepsilon \varepsilon) u_3=M_{st}(u_2) u_3.
\]

Assume $\ell(u_1)\geq 1$. If $u_1u_2\varepsilon\in W^J$, i.e. $u_1u_2\varepsilon =u^J$, we have $(u^J)^{\{s,t\}}=u_1$, $(u^J)_{\{s,t\}}=u_2\varepsilon$, $_{\{s\}}(u_J)=\varepsilon$ and $^{\{s\}}(u_J)=u_3$ and the result is straightforward as in the case $\ell(u_1)=0$.

If $u_1u_2\varepsilon\notin W^J$, there exists $c\in J\cap K\subseteq C_s$ such that $c \in {D}_R(u_1u_2 \varepsilon)$, and we first claim that $c\in {D}_R(u_1)$. If $u_2\varepsilon=e$ this is trivial, otherwise we have  $t\in{D}_R(u_1u_2\varepsilon)$. Note that $s\notin {D}_R(u_1u_2 \varepsilon)$ since otherwise, being $t\in {D}_R(u_1u_2 \varepsilon)$, by a well known fact $u_1u_2 \varepsilon$ would have a reduced expression ending with $\cdots tst \cdots $ ($m_{st}$ factors) and $\{s,t\} \cap  {D}_R(u_1)$ would not be empty.  Hence $c\neq s$. By Lemma \ref{lemma0cap} we have $c\in {D}_R(u_1)$. 

Now if $c$ commutes with both $s$ and $t$ we have $u_1 u_2 u_3=(u_1c) u_2 (cu_3)$, and this triplet still satisfies the conditions of the statement and the result clearly follows by induction. 

So we can assume that $c$ does not commute with $t$. If $u_2\varepsilon=e$, i.e. $u_2\in\{e,s\}$, we have $u_1 u_2 u_3=(u_1{c})(u_2)({c}u_3)$; we observe that $s\notin {D}_R(u_1{c})$ since otherwise $s\in {D}_R(u_1)$ and similarly $s\notin {D}_L({c}u_3)$. If $t \notin  {D}_R(u_1 {c})$, this triplet satisfies the conditions of the statement and the result follows by induction. If $t \in  {D}_R(u_1 {c})$ then $s\not \leq u_1$ by hypothesis, and again  the   result follows by induction by considering the triplet $(u_1{c}t)(tu_2)({c}u_3)$.

We are therefore reduced to the case $u_2\varepsilon \neq e$ and so $t$ is a right descent of $u_2\varepsilon$; hence both $t$ and ${c}$ are right descents  of $u_1 u_2 \varepsilon$. In particular we have a reduced expression for $u_1 u_2 \varepsilon$ which terminates in $t{c}t$ and so $t{c}\leq u_1u_2\varepsilon$; this forces $t\leq u_1$ and so $s\not \leq u_1$.

Now, if $u_2\varepsilon=t$ we let $m=t{c}t\cdots$ ($m_{t,{c}}$ factors) and so $u_1u_2\varepsilon=am$ with $\ell(u_1u_2\varepsilon)=\ell(a)+\ell(m)$ and therefore, since $u_2\varepsilon=t$, we have $u_1=amt$ with $\ell(u_1)=\ell(a)+\ell(m)-1$. Moreover, since $c$ and $t$ are both right descents of $am=u_1u_2\varepsilon$ we have that $\ell(amct)=\ell(am)-2=\ell(u_1)-1$. So 
\[
 u=am \varepsilon u_3=amct\, tc\varepsilon u_3=(amct)(t\varepsilon) (cu_3)=(amct) u_2 (cu_3).
\]
This decomposition of $u$ satisfies our conditions and so we can conclude by induction (as we have already observed,  $\ell(am{c}t)=\ell(u_1)-1$) that
\begin{align*}
 M_{\mathcal R}(u)&=am{c}t M_{st}(u_2) {c} u_3=am{c}t \, t\varepsilon s\, {c} u_3\\& =am \varepsilon s u_3=u_1 u_2 s u_3\\&=u_1 M_{st}(u_2) u_3,
\end{align*}
as, clearly, since $u_2\varepsilon=t$ we have either $u_2=ts$ or $u_2=t$ and in both cases $M_{st}(u_2)=u_2s$.

We are left with the case $st\leq u_2\varepsilon$. We observe that $u_1u_2 \varepsilon$ has a reduced expression that terminates with $t\textrm{-}{c}\textrm{-}t$ and a reduced expression which terminates in $s\textrm{-}t$, and therefore $u_1u_2\varepsilon t$ has a reduced expression which terminates in $s$ and a reduced expression which terminates with $t\textrm{-}c$. To transform one of these two reduced expressions to the other using braid moves we necessarily have to perform a braid relation between $s$ and $t$. Therefore we have that
$tst\cdots$ ($m_{s,t}$ factors) is $\leq u_1u_2\varepsilon t$.  As we already know that $s\not \leq u_1$, we deduce that $u_2\varepsilon t \geq sts\cdots$ ($m_{s,t}-1$ factors). But $u_2\varepsilon t \in W_{s,t}$ and $t\notin D_R(u_2\varepsilon t)$ by construction, so $u_2\varepsilon t=sts\cdots$ ($m_{s,t}-1$ factors). Therefore $u_2\varepsilon=sts\cdots$ ($m_{s,t}$ factors). This is a contradiction since $s\notin {D}_R (u_2\varepsilon)$.
\end{proof}

\begin{thm}\label{special}
 Let $\mathcal R=(J,s,t,M_{st})$ be a right system for $w$. Then $M_{\mathcal R}$  is special.
\end{thm}
\begin{proof}
If $M_{st}\equiv \rho_s$ we have $M_{\mathcal R}(u)=u^Jsu_J$, and the result follows by Theorem~\ref{fa1}. So we can assume that $M_{st}\not \equiv \rho_s$.

 Let $u\lhd v\leq w$ be such that $ u\lhd M_{\mathcal R}(u)\neq v$. By Proposition \ref{nonconfigurazione}, we have to show that $v\lhd M_{\mathcal R}(v)$. Let $v=v_1v_2v_3$, with $v_1=(v^J)^{\{s,t\}}$, $v_2=(v^J)_{\{s,t\}}\cdot \, _{\{s\}}(v_J)$ and $v_3=\,^{\{s\}}(v_J)$. We know that a reduced expression for $u$ can be obtained from any reduced expression of $v$ by deleting one letter. If we consider a reduced expression for $v$ given by the concatenation of reduced expressions of $v_1$, $v_2$ and $v_3$,  we have three cases to consider according to whether such letter comes from $v_1$, $v_2$ or $v_3$.

 Case 1. Let $a\lhd v_1$ be such that $u=av_2v_3$. If $s,t\notin {D}_R(a)$ then the decomposition $a v_2 v_3$ of $u$ satisfies the conditions of Theorem~\ref{fa2} and so $M_{\mathcal R}(u)=aM_{st}(v_2)v_3$. In particular $M_{st}(v_2)\rhd v_2$ and so $M_{\mathcal R}(v)\rhd v$. If there exists $\beta\in \{s,t\}\cap {D}_R(a)$ (such $\beta$ is necessarily unique as $s$ and $t$ cannot be both $\leq a\lhd (v^J)^{\{s,t\}}$ since $M_{st}\not \equiv \rho_s$), then the decomposition $u=(a\beta)(\beta v_2)v_3$ also satisfies the conditions in Theorem~\ref{fa2} and so 
 \[
  M_{\mathcal R}(u)=a\beta M_{st}(\beta v_2)v_3.
 \]
 But since $\beta\leq a$ we also have $\beta\leq (v^J)^{\{s,t\}}\leq (w^J)^{\{s,t\}}$ and so $M_{st}$ commutes with $\lambda_\beta$ on $[e,w_0(s,t)]$ by Property R4. Therefore
 \[
  M_{\mathcal R}(u)=a\beta \beta M_{st}(v_2)v_3=aM_{st}(v_2) v_3
 \]
 and we conclude as in the other case.
 
 Case 2. Let $a\lhd v_2$ be such that $u=v_1 a v_3$; this decomposition automatically satisfies the conditions of Theorem~\ref{fa2}, so $M_{\mathcal R}(u)=v_1 M_{st}(a) v_3$, and the result follows since $M_{st}$ is special.
 
 Case 3. Let $a\lhd v_3$ be such that $u=v_1 v_2 a$. If $s\in {D}_L(a)$ then we can apply Theorem~\ref{fa2} to the decomposition $u=v_1 (v_2s)(sa)$ and obtain $M_{\mathcal R}(u)=v_1 M_{st}(v_2s)sa$. By Lemma \ref{davaw} and Property R5, we have that $M_{st}$ commutes with $\rho_s$, and the result follows. If $s\notin {D}_L(a)$ we can apply directly Theorem~\ref{fa2} to the decomposition $u=v_1v_2 a$ and conclude.
\end{proof}

We can now state the aimed classification theorem of special matchings of lower Bruhat intervals.
\begin{thm}
 Let $(W,S)$ be a Coxeter system and $w\in W$. Then
 \begin{enumerate}
 \item the matching associated with a right or left system of $w$ is special;
  \item a special matching of $w$ is the matching associated with a right or left system of $w$;
  \item if $\mathcal R=(J,s,t,M_{st})$ and $\mathcal R'=(J',s',t',M_{s't'})$ are right systems then $M_\mathcal R=M_{\mathcal R'}$ if and only if $s=s'$, $J\cup C_s=J'\cup C_s$ and one of the following conditions is satisfied:
  \begin{itemize}
   \item $M_{st}(u)=us$ for all $u\leq w_0(s,t)$ and $M_{s't'}(u)=us$ for all $u\leq w_0(s,t')$;
   \item $t=t'$ and $M_{st}=M_{s't'}$
  \end{itemize}
  \item if $\mathcal R=(J,s,t,M_{st})$ is a right system and $\mathcal L=(K,s',t',M_{s't'})$ is a left system then $M_\mathcal R=\,_\mathcal L M$ if and only if $s=s'$, $J\cap K \subseteq C_s$, $J\cup K \subseteq S\setminus C_s$, $M_{st}=\rho_s$, $M_{s't'}=\lambda_s$.
\end{enumerate}
\end{thm}
\begin{proof}
 This follows immediately from Lemma \ref{diheequa}, Theorem \ref{caratteri}, Corollary \ref{charpar}, and Theorem \ref{special}.
\end{proof}
\begin{ack}
 We thank the referees for their comments and their careful reading of the manuscript.
\end{ack}

\end{document}